\documentclass[11pt]{article} 

\usepackage{amssymb}
\usepackage{tikz}
\usepackage{amsthm}
\usepackage{amsmath, epsfig}
\usepackage{graphicx}
\usepackage{mathrsfs}
\usepackage{color}
\usepackage{multicol}
\usepackage{mathrsfs}
\usepackage{float}
\usepackage{comment}
\usepackage{pgf,tikz}

\usepackage{breqn}

\usetikzlibrary{arrows}

\newtheorem{theorem}{Theorem} 
\newtheorem{lem}{Lemma} 
\newtheorem{prop}{Proposition} 
\newtheorem{conj}{Conjecture} 
\newtheorem{fact}{Fact}

\newtheorem{claim}{\bf Claim}
\newtheorem{subclaim}{Subclaim}
\newtheorem{sub-subclaim}{Subclaim}[subclaim]

\newtheoremstyle{def}{}{}{}{}{}{.}{ }{}
\theoremstyle{def}

\newtheoremstyle{case}{}{}{}{}{}{.}{ }{}
\theoremstyle{case}

\newcommand{\fon}{\fontsize{9pt}{\baselineskip}\selectfont}

\begin{document}

\begin{center}
{\Large On Degree Sum Conditions and Vertex-Disjoint Chorded Cycles }
\end{center}

\begin{center}
\renewcommand{\thefootnote}{\fnsymbol{footnote}}
Bradley Elliott\,$^{a}$ \hspace{0.3cm}
Ronald J. Gould\,$^{a}$ \hspace{0.3cm} 
Kazuhide Hirohata\,$^{b}$
\footnote[0]{{\it E-mail addresses:} bradley.elliott@emory.edu\,(B.\,Elliott),
rg@emory.edu\,(R.J.\,Gould), hirohata@ece.ibaraki-ct.ac.jp (K.\,Hirohata). }

$^{a}$\,{\fon Department of Mathematics, Emory University, 
Atlanta, GA 30322 USA.}\\
$^{b}$\,{\fon Department of Industrial Engineering, Computer Science, 
National Institute of Technology, \\
Ibaraki College, Hitachinaka, 312-8508 Japan.}
\end{center}

\begin{abstract}
In this paper, we consider a general degree sum condition sufficient to 
imply the existence of $k$ vertex-disjoint chorded cycles in a graph $G$. 
Let $\sigma_t(G)$ be the minimum degree sum of $t$ independent vertices of $G$. 
We prove that if $G$ is a graph of sufficiently large order and 
$\sigma_t(G)\geq 3kt-t+1$ with $k\geq 1$, then $G$ contains $k$ 
vertex-disjoint chorded cycles. 
We also show that the degree sum condition on $\sigma_t(G)$ is sharp.
To do this, we also investigate graphs without chorded cycles.

\vspace{0.2cm}

\end{abstract}

\section{Introduction}

\hspace*{0.5cm} 
The study of cycles in graphs is a rich and an important area.  
One question of particular interest is to find conditions that guarantee 
the existence of $k$ vertex-disjoint cycles. 
Let $G$ be a graph. 
Corr\'adi and Hajnal~\cite{CH} first considered a minimum degree condition 
to imply a graph must contain $k$ vertex-disjoint cycles, proving that 
if $|G|\geq 3k$ and the minimum degree $\delta(G) \geq 2k$, 
then $G$ contains $k$ vertex-disjoint cycles.  For an integer 
$t\geq 1$, let $$\sigma_t (G)=\mbox{min}\left\{\sum_{v\in X}\deg_G(v)\, :\,   
X\ {\rm is\ an\ independent\ vertex\ set\ of}\ G\ {\rm with}\ |X|=t.\right\},$$  
and $\sigma_t (G)=\infty$ when the independence number is $t-1$ or less. 
Enomoto~\cite{E} and Wang~\cite{W} independently extended the 
Corr\'adi and Hajnal result, requiring a weaker condition on the 
minimum degree sum of any two non-adjacent vertices. 
They proved that if $|G|\geq 3k$ and $\sigma_2(G)\geq 4k-1$, then $G$ 
contains $k$ vertex-disjoint cycles.  In 2006, Fujita et al.~\cite{FMTY} 
proved that if $|G| \geq 3k+2$ and $\sigma_3 (G) \geq 6k-2$, 
then $G$ contains $k$ vertex-disjoint cycles, and in~\cite{GHK}, 
this result was extended to $\sigma_4(G)\geq 8k-3$.   
Recently, Ma and Yan~\cite{MaYan} proved a conjecture from~\cite {GHK} by showing that
if $G$ has sufficiently large order and $\sigma_t (G) \geq 2kt-t+1$,
then $G$ contains $k$ vertex-disjoint cycles.
	
A {\it chord} of a cycle is an edge between two non-consecutive vertices of 
the cycle.  An extension of the study of vertex-disjoint cycles is that 
of vertex-disjoint chorded cycles. 
We say a cycle is \emph{chorded} if it contains at least one chord. 
In 2008, Finkel~\cite{F} proved the following result on the existence 
of $k$ vertex-disjoint chorded cycles which can be viewed as an extension
of the Corr\'adi and Hajnal result. 

\begin{theorem}
[\rm Finkel \cite{F}] 
Let $k\geq 1$ be an integer. 
If $G$ is a graph of order at least $4k$ and $\delta(G) \geq 3k$,
then $G$ contains $k$ vertex-disjoint chorded cycles. 
\label{Finkel} 
\end{theorem}

In 2010, Chiba et al. \cite{CFGL} extended the above result by using 
the $\sigma_2(G)$ condition.   

\begin{theorem}
[\rm Chiba, Fujita, Gao, Li \cite{CFGL}] 
Let $k\geq 1$ be an integer. 
If $G$ is a graph of order at least {\rm 4}k and $\sigma_2(G)\geq 6k-1$, 
then $G$ contains $k$ vertex-disjoint chorded cycles.
\label{cfgl}
\end{theorem}

Recently, Theorem \ref{cfgl} was extended as follows. 

\begin{theorem}
[\rm \cite{GHK2}] 
Let $k\geq 1$ be an integer. 
If $G$ is a graph of order at least $8k+5$ and $\sigma_3(G) \geq 9k-2$,
then $G$ contains $k$ vertex-disjoint chorded cycles. 
\label{sigma3} 
\end{theorem}

The last result was further extended to $\sigma_4 (G)$ in \cite{sigma4}.

\begin{theorem} [\rm \cite{sigma4}]
Let $k\geq 1$ be an integer. 
If $G$ is a graph of order $n \geq 11k+7$ and $\sigma_4 (G) \geq 12k - 3$,
then $G$ contains $k$ vertex-disjoint chorded cycles. 
\label{k ge 2} 
\end{theorem}

In this paper, we prove the following result in Section 4.

\begin{theorem}
For $k\geq 1$ and $t\geq1$,
if $G$ is a graph of order $n \geq (10t -1)(k-1) +12t+13$ and $\sigma_t(G) \geq 3kt - t +1$,
then $G$ contains $k$ vertex-disjoint chorded cycles.  Further,
this degree condition is sharp.
\label{thm:main}
\end{theorem}

\noindent
{\it Remark.}  
To see the sharpness of the degree condition of Theorem \ref{thm:main}, 
for $n$ sufficiently large order, consider the complete bipartite graph 
$B = K_{3k-1, n - 3k + 1}$.  
Then $\sigma_t (B) = t(3k-1)$.  
Further, it is not possible to construct $k$ vertex-disjoint chorded cycles in $B$, 
as any chorded cycle must use three vertices from the partite set of order $3k-1$.

\vspace{0.25cm} 

All graphs considered here are simple and undirected finite.   
For terms not defined here see \cite{G}.  Let $G$ be a graph. 
Let $H$ be a subgraph of $G$, and let $S\subseteq V(G)$.
For $u\in V(G)$, the set of neighbors of $u$ in $G$ is denoted by $N_G(u)$, 
and we denote $\deg_G(u)=|N_G(u)|$ and  $N_H(u)=N_G(u)\cap V(H)$.
For $u\in V(G)-S$, $N_S(u)=N_G(u)\cap S$.  
Furthermore, $N_G(S)=\bigcup_{w\in S}N_G(w)$ and $N_H(S)=N_G(S)\cap V(H)$. 
Let $A, B$ be two vertex-disjoint subgraphs of $G$. 
Then $N_G(A)=N_G(V(A))$ and $N_B(A)=N_G(A)\cap V(B)$.   
The subgraph of $G$ induced by $S$ is denoted by $\langle S \rangle$. 
Let $G-S=\langle V(G)-S\rangle$ and 
$G-H=\langle V(G)-V(H)\rangle$. 
If $S=\{u\}$, then we write $G-u$ for $G-S$. 
If there is no fear of confusion, then we use the same symbol 
for a graph and its vertex set. 
For two disjoint graphs $G_1$ and $G_2$, $G_1\cup G_2$ denotes the disjoint
union of $G_1$ and $G_2$.  Let $Q$ be a path or a cycle with a 
given orientation and $x\in V(Q)$.  
Then $x^+$ denotes the first successor of $x$ on $Q$ and 
$x^-$ denotes the first predecessor of $x$ on $Q$. 
If $x, y\in V(Q)$, then $Q[x,y]$ denotes the path of $Q$ from $x$ to $y$ 
(including $x$ and $y$) in the given direction. 
The reverse sequence of $Q[x,y]$ is denoted by $Q^-[y,x]$. 
We also write $Q(x,y]=Q[x^+,y]$, $Q[x, y)=Q[x, y^-]$ and $Q(x, y)=Q[x^+, y^-]$. 
If $Q$ is a path (or a cycle), say $Q = x_1,x_2, \ldots ,x_t (,x_1)$, then 
we assume that an orientation of $Q$ is given from $x_1$ to $x_t$. 
If $P$ is a path connecting $x$ and $y$, 
then we denote the path $P$ with an orientation from $x$ to $y$ as $P[x,y]$. 
The reverse sequence of $P[x,y]$ is denoted by $P^-[y,x]$.  
For $X\subseteq V(G)$, let $\deg_H(X)=\sum_{x\in X} \deg_H(x)$. 
If $H=G$, then we denote $\deg_G(X)=\deg_H(X)$. 
If $G$ is one vertex, that is, $V(G)=\{x\}$, then we simply write $x$ instead of $G$. 
For an integer $r\geq 1$ and two vertex-disjoint subgraphs $A, B$ of $G$, 
we denote by $(d_1, d_2, \ldots , d_r)$ a degree sequence from $A$ to $B$ such that 
$\deg_B(v_i)\geq d_i$ and $v_i\in V(A)$ for each $1\leq i\leq r$.  
In this paper, since it is sufficient to consider the case of equality in the above inequality, 
when we write $(d_1, d_2, \ldots , d_r)$, we assume $\deg_B(v_i)=d_i$ for each $1\leq i\leq r$.    
For $X,Y\subseteq V(G)$, $E(X,Y)$ denotes the set of edges of $G$ connecting a vertex in $X$ 
and a vertex in $Y$. 
A cycle of length $\ell$ is called a {\it $\ell$-cycle}. 
For a graph $G$, $comp(G)$ is the number of components of $G$. 
Let $R$ be a graph. 
If $G$ has no induced subgraph isomorphic to $R$, then $G$ is called {\it R-free}.  

\section{Graphs with No Chorded Cycles}

\hspace*{0.5cm} 
In this section, we examine some useful properties of graphs that contain no chorded cycles. 
Our ultimate goal is to show they contain large independent sets of small degree sum. 
This will be important in our proof later.

\begin{lem} 
Let $T$ be a tree of order $n\geq 2$. 
Then the following statements hold. 

\vspace{0.2cm}

\noindent
$(i)$ $T$ has at least $n/2+1$ vertices of degree at most $2$. 

\vspace{0.2cm}

\noindent
$(ii)$ $T$ contains an independent set $I$ of order at least $n/4$ with each vertex of $I$ 
having degree at most $2$ in $T$.
\label{tree}
\end{lem}

\begin{proof}
Let $\{v_1, \ldots, v_b\}$ be the set of branch vertices in $T$. 
Let $\ell$ be the number of leaves in $T$ and $s$ be the number of stem vertices. 
Clearly $\ell + s + b = n$.  
Since $T$ has $n-1$ edges, the degree sum of $T$ is 
\begin{align*}
2(n-1) = \ell + 2s + \sum_{i=1}^b \deg_T(v_i) \geq \ell + 2(n-\ell-b) + 3b,
\end{align*}
which implies $\ell \geq b+2$. 
Consequently, 
\begin{align*}
\ell+s \geq (b+2)+s &= (b+s)+2=(n-\ell)+2 \\
2\ell +s &\geq n+2 \\
\ell + \frac{s}{2} &\geq \frac{n}{2}+1.
\end{align*}
If $L$ is the set of all leaves and stems in $T$, then 
\begin{align*}
|L| = \ell + s \geq \ell + \frac{s}{2} \geq \frac{n}{2}+1. 
\end{align*}
Thus $(i)$ holds. 

Since $T$ is bipartite, one of the partite sets contains at least half the vertices of $L$. 
Thus $T$ contains an independent subset $I\subset L$ with $|I| \geq |L|/2 \geq  n/4$, and 
$(ii)$ holds. 
\end{proof}

\begin{lem} 
If $H$ is a non-chorded graph of order $n$, then $H$ contains an independent set 
$I$ of order at least $n/12$ with each vertex of $I$ having degree at most $2$ in $H$.
\label{indep}
\end{lem}

Before proving Lemma \ref{indep}, we state and prove some helpful propositions.

\vspace{0.25cm}

\noindent
{\bf Definitions.} 
A {\it biconnected graph} is a non-separable graph. 
Note that any two vertices (two edges) of a biconnected graph lie on a common cycle.  
A {\it non-chorded} graph is a graph not containing any chorded cycles. 
A {\it leaf} is a vertex of degree 1. 
A {\it stem} is a vertex of degree 2. 
A {\it branch} is a vertex of degree at least 3. 

\begin{prop}
\label{triangle}
Every non-chorded biconnected graph $H$ of order at least four is triangle-free.
\end{prop}

\begin{proof}
Suppose $H$ contains a triangle on vertices $a,b,c$. 
Since $H$ is connected, without loss of generality, we can say $a$
has some neighbor $d\in V(H)-\{b,c\}$. 
Since $H$ is biconnected, edges $ab$ and $ad$ must lie on a common cycle in $H$.  
Let $C$ be such a cycle. 
If $C$ contains edge $bc$, then $ac$ is a chord on the cycle, a contradiction.  
If $C$ does not contain $bc$, then $\langle C \cup c\rangle$ contains 
a cycle with chord $ab$, a contradiction.
\end{proof}

\begin{prop} 
\label{decomp}
Let $k\geq 1$ be an integer. 
If $H$ is a non-chorded biconnected graph of order at least four, then $E(H)$ 
can be decomposed into 

\begin{itemize}
\item 
a cycle $C = F_0$, and 
\item 
if $C$ is not a hamiltonian cycle in $H$, 
then a sequence of paths $P_1, \ldots, P_k$ $($each with at least two edges$)$ 
where the endpoints of $P_i$ are $a_i$, $b_i$ $(a_i\ne b_i)$, 
\end{itemize}

\noindent
such that there exists a sequence of subgraphs $F_1, \ldots, F_k$ of $H$, 
where for all $1\leq i\leq k$, 

\vspace{0.2cm}

\noindent
$(i)$ $F_i = P_i \cup F_{i-1}$, 

\vspace{0.2cm}

\noindent
$(ii)$ $V(P_i) \cap V(F_{i-1}) = \{a_i, b_i\}$, 

\vspace{0.2cm}

\noindent
$(iii)$ $F_i$ is a non-chorded biconnected graph, and

\vspace{0.2cm}

\noindent
$(iv)$ $F_k = H$. 
\end{prop} 


\begin{proof}
Let $C$ be a cycle in $H$. 
Note that $H$ is triangle-free by Proposition \ref{triangle}, and in particular, 
$C$ is not a triangle.  
Let $F_0 = C$ and let $E_1 = E(H) \setminus E(F_0)$.  
If $C$ is a hamiltonian cycle in $H$, then since $H$ is non-chorded, $E_1=\emptyset$. 
For each $i\geq 1$, if $E_i\ne \emptyset$, do the following:
Select any $f\in E(F_{i-1})$ and any $e_i\in E_i$. 
Since $H$ is biconnected, there exists a cycle $C_i$ in $H$ containing $f$ and $e_i$. 
Let $P_i$ be a path in $C_i$ containing $e_i$ 
so that the endpoints of $P_i$ are in $V(F_{i-1})$. 
Note that $|E(P_i)| \geq 2$. 
Call these endpoints $a_i,b_i$, and assume that $ V(P_i) \cap V(F_{i-1}) = \{a_i, b_i\}$. 
Let $F_i = P_i \cup F_{i-1}$. 
Since $H$ is non-chorded biconnected, the graph $F_i$ is also non-chorded biconnected. 
Let $E_{i+1} = E_i \setminus E(P_i)$.  
Let $k+1$ be the minimum index so that $E_{k+1}$ is empty.  Then $F_k = H$.
\end{proof}

\begin{prop}
\label{stem1}
Let $k\geq 1$ be an integer. 
Let $C=F_0$ be any cycle of order at least four, let $P_1, \ldots, P_k$ be a sequence of paths 
$($each with at least two edges$)$ such that for each $1\leq i\leq k$, 
$P_i$ is a path from $a_i$ to $b_i$ {\rm (}$a_i\ne b_i${\rm )}, and 
let $F_1, \ldots, F_k$ be a sequence of graphs such that for each $1\leq i\leq k$, 

\vspace{0.2cm}

\noindent
$(i)$ $F_i = P_i \cup F_{i-1}$, 

\vspace{0.2cm}

\noindent
$(ii)$ $V(P_i) \cap V(F_{i-1}) = \{a_i, b_i\}$, and 

\vspace{0.2cm}

\noindent
$(iii)$ $F_i$ is a non-chorded biconnected graph. 

\vspace{0.2cm}


\noindent
Then for each $1\leq i\leq k$, there exists some vertex $v\in P_i(a_i, b_i)$ 
such that $\deg_{F_k}(v)=2$. 
Further, there exist two distinct vertices $x,x'\in V(C) \setminus \bigcup_{i=1}^k V(P_i)$ 
such that $\deg_{F_k}(x) = \deg_{F_k}(x') = 2$. 
\end{prop}

\begin{proof}
Suppose for a contradiction that for some $1\leq \ell \leq k$, 
$\deg _{F_k}(v)\geq 3$ for all $v\in P_\ell(a_\ell, b_\ell)$.  
Let $P_\ell : v_0=a_{\ell}, v_1, \ldots, v_{t-1}, v_t=b_{\ell}$, and 
let $F=F_{\ell-1}\setminus \{v_0,v_t\}$. 
Note that since $F_k$ is non-chorded biconnected graph of order at least four, 
$F_k$ is triangle-free by Proposition \ref{triangle}. 

\begin{claim}
\label{paths}
For each $1\leq i\leq t-2$, there exists a path $S_i$ in $F_k$ from $v_i$ to $v_j$ for some 
$i+2\leq j\leq t$ such that $S_i(v_i,v_j) \cap V(P_{\ell})=\emptyset$ and 
$V(S_i)\cap V(F)=\emptyset$.  
\end{claim}

\begin{proof}
We prove Claim \ref{paths} by induction. 
Since $\deg _{F_k}(v)\geq 3$ for all $v\in P_\ell(a_\ell, b_\ell)$,  
there exists a neighbor $u_i$ of $v_i$ with $u_i \not\in \{v_{i-1}, v_{i+1}\}$ 
for each $1\leq i\leq t-2$. 
Since $F_k$ is biconnected, there exists a path $S_i$ in $F_k$ starting with $v_i, u_i,\ldots,$ 
terminating at $v_j$ with $i\ne j$ such that $S_i(v_i,v_j)\cap V(P_{\ell})=\emptyset$. 

First we prove the case where $i=1$. 
Suppose $V(S_1)\cap V(F) \ne \emptyset$.  
Then there exists a vertex $w\in V(S_1)\cap V(F)$ such that 
$S_1(v_1,w)\cap V(F)=\emptyset$. 
Since $F_{\ell-1}$ is biconnected, there exists a cycle $C_1 \subseteq F_{\ell-1}$ containing $v_0$ and $w$.  
We assume that an orientation of $C_1$ is given from $v_0$ to $w$. 
Suppose $v_t \in V(C_1)$, so $v_t \in C_1(v_0,w)$ or $v_t \in C_1(w,v_0)$.  
Without loss of generality, we may assume that $v_t \in C_1(v_0,w)$. 
Then $$P_\ell[v_1,v_t],C_1^-[v_t,v_0],C_1^-[v_0,w],S_1^-[w,v_1]$$ 
is a cycle with chord $v_0v_1$, a contradiction (see Figure 1). 

\begin{figure}[hbpt]
\centering
\definecolor{ccqqqq}{rgb}{0.8,0,0}
\begin{tikzpicture}[line cap=round,line join=round,>=triangle 45,x=1cm,y=1cm]
\clip(-7.694545454545467,-4.741322314049571) rectangle (7.908760330578507,1.6057851239669338);
\draw [line width=1.2pt] (-7,0)-- (-2.5,0);
\draw [line width=1.2pt] (-7,0)-- (-7,-3.5);
\draw [line width=1.2pt] (-7,-3.5)-- (-2.5,-3.5);
\draw [line width=1.2pt] (-2.5,-3.5)-- (-2.5,0);
\draw [shift={(-4.478693982191317,-1.75)},line width=1.2pt,color=ccqqqq]  plot[domain=-0.6538690060038537:0.6538690060038537,variable=\t]({1*1.2330214559274704*cos(\t r)+0*1.2330214559274704*sin(\t r)},{0*1.2330214559274704*cos(\t r)+1*1.2330214559274704*sin(\t r)});
\draw [shift={(-4.408436048203764,-1.77592402052921)},line width=1.2pt,color=ccqqqq]  plot[domain=0.7068878401935449:2.4380383163243016,variable=\t]({1*1.194702531725067*cos(\t r)+0*1.194702531725067*sin(\t r)},{0*1.194702531725067*cos(\t r)+1*1.194702531725067*sin(\t r)});
\draw [shift={(-4.360386786435848,-1.691547145743391)},line width=1.2pt]  plot[domain=2.5189536633563394:5.528896889538317,variable=\t]({1*1.180619091760362*cos(\t r)+0*1.180619091760362*sin(\t r)},{0*1.180619091760362*cos(\t r)+1*1.180619091760362*sin(\t r)});
\draw [line width=1.2pt] (-3.5,-1)-- (-2,-1);
\draw [line width=1.2pt,color=ccqqqq] (-3.5,-2.5)-- (-1,-2.5);
\draw [line width=1.2pt,color=ccqqqq] (-2,-1)-- (-1,-1);
\draw [shift={(-1.3115516444066975,-1.75)},line width=1.2pt,color=ccqqqq]  plot[domain=-1.1770831123163763:1.177083112316376,variable=\t]({1*0.8121357196506734*cos(\t r)+0*0.8121357196506734*sin(\t r)},{0*0.8121357196506734*cos(\t r)+1*0.8121357196506734*sin(\t r)});
\draw [shift={(-3.6600839858827845,-0.6105356966537477)},line width=1.2pt,color=ccqqqq]  plot[domain=-0.2304377709326193:3.373857579996412,variable=\t]({1*1.7051572607139371*cos(\t r)+0*1.7051572607139371*sin(\t r)},{0*1.7051572607139371*cos(\t r)+1*1.7051572607139371*sin(\t r)});
\draw (-4.058181818181829,-0.9066115702479323) node[anchor=north west] {$v_0$};
\draw (-4.02512396694216,-2.05) node[anchor=north west] {$v_t$};
\draw (-2.2234710743801753,-1.0223140495867749) node[anchor=north west] {$v_1$};
\draw (-5.338264462809929,-0.9066115702479323) node[anchor=north west] {$w$};
\draw (-3.959008264462821,1.0925619834710677) node[anchor=north west] {$S_1$};
\draw (-1.1813223140495964,-1.5016528925619796) node[anchor=north west] {$P_{\ell}$};
\draw (-6.154876033057864,-1.6338842975206567) node[anchor=north west] {$C_1$};
\draw (-7.000330578512409,-2.907438016528917) node[anchor=north west] {$F_{\ell-1}$};
\draw [line width=1.2pt] (0.4818236459258367,-0.018058875759849435)-- (4.981823645925837,-0.018058875759849435);
\draw [line width=1.2pt] (0.4818236459258367,-0.018058875759849435)-- (0.4818236459258367,-3.5180588757598494);
\draw [line width=1.2pt] (0.4818236459258367,-3.5180588757598494)-- (4.981823645925837,-3.5180588757598494);
\draw [line width=1.2pt] (4.981823645925837,-3.5180588757598494)-- (4.981823645925837,-0.018058875759849435);
\draw [shift={(3.493888678760727,-1.7680588757598497)},line width=1.2pt,color=ccqqqq]  plot[domain=-0.9940135121533284:0.9940135121533287,variable=\t]({1*0.8947516595024663*cos(\t r)+0*0.8947516595024663*sin(\t r)},{0*0.8947516595024663*cos(\t r)+1*0.8947516595024663*sin(\t r)});
\draw [shift={(3.0723047107937878,-1.144284396613474)},line width=1.2pt,color=ccqqqq]  plot[domain=0.1379018723786176:3.0070242841392294,variable=\t]({1*0.9182361218545616*cos(\t r)+0*0.9182361218545616*sin(\t r)},{0*0.9182361218545616*cos(\t r)+1*0.9182361218545616*sin(\t r)});
\draw [line width=1.2pt] (3.9818236459258367,-1.0180588757598494)-- (5.481823645925836,-1.0180588757598494);
\draw [line width=1.2pt,color=ccqqqq] (3.9818236459258367,-2.5180588757598494)-- (6.481823645925836,-2.5180588757598494);
\draw [line width=1.2pt,color=ccqqqq] (5.481823645925836,-1.0180588757598494)-- (6.481823645925836,-1.0180588757598494);
\draw [shift={(6.170272001519139,-1.7680588757598494)},line width=1.2pt,color=ccqqqq]  plot[domain=-1.1770831123163772:1.1770831123163772,variable=\t]({1*0.8121357196506731*cos(\t r)+0*0.8121357196506731*sin(\t r)},{0*0.8121357196506731*cos(\t r)+1*0.8121357196506731*sin(\t r)});
\draw [shift={(3.821739660043051,-0.628594572413597)},line width=1.2pt,color=ccqqqq]  plot[domain=-0.2304377709326193:3.3738575799964123,variable=\t]({1*1.7051572607139358*cos(\t r)+0*1.7051572607139358*sin(\t r)},{0*1.7051572607139358*cos(\t r)+1*1.7051572607139358*sin(\t r)});
\draw (3.4294214876032982,-0.9231404958677669) node[anchor=north west] {$v_0$};
\draw (3.4624793388429675,-2.05) node[anchor=north west] {$v_t$};
\draw (5.264132231404952,-1.0388429752066095) node[anchor=north west] {$v_1$};
\draw (2.1593388429751986,-0.9066115702479323) node[anchor=north west] {$w$};
\draw (3.512066115702472,1.0925619834710677) node[anchor=north west] {$S_1$};
\draw (6.306280991735531,-1.518181818181814) node[anchor=north west] {$P_{\ell}$};
\draw (1.6277685950413139,-1.354545454545452) node[anchor=north west] {$C_1$};
\draw (0.4872727272727184,-2.907438016528917) node[anchor=north west] {$F_{\ell-1}$};
\draw [shift={(3.072235607437928,-1.1028245309046378)},line width=1.2pt]  plot[domain=-3.2311820308662256:0.09292288020448601,variable=\t]({1*0.9135292091949807*cos(\t r)+0*0.9135292091949807*sin(\t r)},{0*0.9135292091949807*cos(\t r)+1*0.9135292091949807*sin(\t r)});
\draw [shift={(3.506767734126701,-1.7680588757598494)},line width=1.2pt]  plot[domain=1.0061737278151013:5.277011579364485,variable=\t]({1*0.8877939622093112*cos(\t r)+0*0.8877939622093112*sin(\t r)},{0*0.8877939622093112*cos(\t r)+1*0.8877939622093112*sin(\t r)});
\draw (2.5533884297520584,-2.4933884297520583) node[anchor=north west] {$C_2$};
\draw (-5.9920661157024915,-3.933884297520649) node[anchor=north west] {Figure 1:};
\draw (0.601322314049578,-3.933884297520649) node[anchor=north west] {Figure 2:};
\draw (-4.28710743801654,-3.888884297520649) node[anchor=north west] {$v_t \in C_1(v_0,w)$};
\draw (2.289752066115694,-3.848884297520649) node[anchor=north west] {$C_2^-(v_t,v_0)\cap C_1(v_0,w)=\emptyset$};
\begin{scriptsize}
\draw [fill=black] (-3.5,-1) circle (2.5pt);
\draw [fill=black] (-3.5,-2.5) circle (2.5pt);
\draw [fill=black] (-5.319453589235699,-1.0030325797418178) circle (2.5pt);
\draw [fill=black] (-2,-1) circle (2.5pt);
\draw [fill=black] (3.9818236459258367,-1.0180588757598494) circle (2.5pt);
\draw [fill=black] (3.9818236459258367,-2.5180588757598494) circle (2.5pt);
\draw [fill=black] (2.162370056690138,-1.0210914555016672) circle (2.5pt);
\draw [fill=black] (5.481823645925836,-1.0180588757598494) circle (2.5pt);
\end{scriptsize}
\end{tikzpicture}
\end{figure}

Thus $v_t \not\in C_1(v_0,w)$. Similarly $v_t \not
\in C_1(v_0,w)$, hence $v_t \not\in V(C_1)$. 
Since $F_{\ell-1}$ is biconnected, there exists a cycle $C_2$ in $F_{\ell-1}$ containing $v_0$ and $v_t$.  
We assume that an orientation of $C_2$ is given from $v_0$ to $v_t$. 
Without loss of generality, we may assume that $w \not\in C_2^-(v_t,v_0)$. 
If $C_2^-(v_t,v_0)\cap V(C_1)=\emptyset$, then 
$$P_\ell[v_1,v_t],C_2^-[v_t,v_0],C_1^-[v_0,w],S_1^-[w,v_1]$$ 
is a cycle with chord $v_0v_1$, a contradiction (see Figure 2). 
Thus we may assume that $C_2^-(v_t,v_0)\cap V(C_1)\ne \emptyset$. 
Let $z$ be a vertex such that $z\in C_2^-(v_t,v_0) \cap V(C_1)$ and $C_2^-(v_t,z) \cap V(C_1) = \emptyset$. 
By assumption, $z\neq w$. 
If $z\in C_1(v_0,w)$, then 
$$P_\ell[v_1,v_t],C_2^-[v_t,z],C_1^-[z,v_0],C_1^-[v_0,w],S_1^-[w,v_1]$$ 
is a cycle with chord $v_0v_1$, a contradiction. 
Otherwise, $z\in C_1^-(v_0,w)$, and similarly
$$P_\ell[v_1,v_t],C_2^-[v_t,z],C_1[z,v_0],C_1[v_0,w],S_1^-[w,v_1]$$ 
is a cycle with chord $v_0v_1$, a contradiction. 
Thus $V(S_1) \cap V(F)=\emptyset$. 
Next suppose $j\in \{0,2\}$, that is, $v_j\in \{v_0,v_2\}$. 
If $j=0$, then $$P_\ell[v_1,v_t], C_2[v_t,v_0],S_1^-[v_0,v_1]$$ 
is a cycle with chord $v_0v_1$, a contradiction.  
If $j=2$, then similarly, we can find a cycle with chord $v_1v_2$, a contradiction.   

For induction, assume that Claim \ref{paths} is true for $i-1$. 
Thus there exists a path $S_{i-1}$ in $F_k$ from $v_{i-1}$ to $v_{j'}$ for some $i+1\leq j'\leq t$ 
satisfying the conditions of Claim \ref{paths}.  
Suppose that every path $S_i$ starting at $v_i, u_i, \ldots$ passes through some vertex 
$x \in V(F) \cup S_{i-1}(v_{i-1},v_{j'})$ 
before reaching any $v_j$ with $i\ne j$.  
Then select a vertex $x$ such that $S_i(v_i,x)\cap (V(F)\cup S_{i-1}(v_{i-1},v_{j'}))=\emptyset$.  
First suppose $x\in V(F)$. 
Since $F_{\ell-1}$ is connected, there exists a path $Q_1$ in $F_{\ell-1}$ from $x$ to $v_0$.   
Then 
$$P_\ell[v_0,v_{i-1}],S_{i-1}[v_{i-1},v_{j'}],P^-_\ell[v_{j'},v_i],S_i[v_i,x],Q_1[x,v_0]$$ 
is a cycle with chord $v_{i-1}v_i$, a contradiction. 
Next suppose $x\in S_{i-1}(v_{i-1},v_{j'})$. 
Since $F_{\ell-1}$ is connected, there exists a path $Q_2$ in $F_{\ell-1}$ from $v_t$ to $v_0$.  
Then 
$$P_\ell[v_0,v_{i-1}],S_{i-1}[v_{i-1},x],S^-_i[x,v_i],P_\ell[v_i,v_t],Q_2[v_t,v_0]$$ 
is a cycle with chord $v_{i-1}v_i$, a contradiction. 
Thus $S_i$ is a path from $v_i$ to $v_j$ not containing any vertex in  
$V(F)\cup S_{i-1}(v_{i-1},v_{j'})$.  
If $j\geq i+2$, then Claim \ref{paths} holds. 
Thus we may assume that $j\leq i+1$. 
Suppose $j\leq i-2$. 
Then 
$$P_\ell[v_j,v_{i-1}],S_{i-1}[v_{i-1},v_{j'}],P^-_\ell[v_{j'},v_i],S_i[v_i,v_j]$$ 
is a cycle with chord $v_{i-1}v_i$, a contradiction. 
If $j=i-1$, then using the above path $Q_2$,  
$$P_\ell[v_0,v_{i-1}],S_i^-[v_{i-1},v_i],P_\ell[v_i,v_t],Q_2[v_t,v_0]$$ 
is a cycle with chord $v_{i-1}v_i$, a contradiction.  
If $j=i+1$, then similarly, we can find a cycle with chord $v_iv_{i+1}$, a contradiction. 
Thus for each $1\leq i\leq t-2$, 
there exists a path $S_i$ in $F_k$ from $v_i$ to $v_j$ for some $i+2\leq j\leq t$ 
satisfying the conditions of Claim \ref{paths}.  
\end{proof}

By Claim \ref{paths}, there exists a path $S_{t-2}$ from $v_{t-2}$ to $v_t$ such that 
$S_{t-2}(v_{t-2},v_{t}) \cap V(P_{\ell})=\emptyset$ and $V(S_{t-2})\cap V(F)=\emptyset$.  
Since $\deg_{F_k}(v_{t-1})\geq 3$ by our assumption, 
there exists a neighbor $u_{t-1}$ of $v_{t-1}$ with $u_{t-1} \not\in \{v_{t-2}, v_t\}$. 
Since $F_k$ is biconnected, there exists a path $S_{t-1}$ in $F_k$ starting with 
$v_{t-1},u_{t-1},\ldots,$ terminating at $v_j$ with $j\ne t-1$ such that 
$S_{t-1}(v_{t-1},v_j)\cap V(P_{\ell})=\emptyset$. 
Since $F_{\ell-1}$ is biconnected, there exists a cycle $C_1$ containing $v_0$ and $v_t$.  
We assume that an orientation of $C_1$ is given from $v_0$ to $v_t$.  
Suppose that every path $S_{t-1}$ starting at $v_{t-1}$ passes through some vertex 
$x \in V(F) \cup S_{t-2}(v_{t-2},v_t)$ before reaching $v_j$ with $j\neq t-1$. 
Then we take a vertex $x$ such that 
$S_{t-1}(v_{t-1},x)\cap (V(F)\cup S_{t-2}(v_{t-2},v_{t}))=\emptyset$.  
First suppose $x\in V(F)$. 
Since $F_{\ell-1}$ is biconnected, two vertices $v_0$ and $x$ must lie on a common cycle 
$C_2$ in $F_{\ell-1}$. 
We assume that an orientation of $C_2$ is given from $v_0$ to $x$.  
Then we may assume that $v_t\not\in C_2(x,v_0)$. 
Thus 
$$P_\ell[v_0,v_{t-2}],S_{t-2}[v_{t-2},v_t],P^-_\ell[v_t,v_{t-1}],S_{t-1}[v_{t-1},x],C_2[x,v_0]$$ 
is a cycle with chord $v_{t-2}v_{t-1}$, a contradiction. 
Next suppose $x\in S_{t-2}(v_{t-2},v_{t})$. 
Then 
$$P_\ell[v_0,v_{t-2}],S_{t-2}[v_{t-2},x],S^-_{t-1}[x,v_{t-1}],P_\ell[v_{t-1},v_t], C_1[v_t,v_0]$$
is a cycle with chord $v_{t-2}v_{t-1}$, a contradiction. 
Thus $S_{t-1}$ is a path from $v_{t-1}$ to $v_j$ not containing any vertex in 
$V(F)\cup S_{t-2}(v_{t-2},v_{t})$. 
If $j\leq t-2$, then 
$$P_\ell[v_j,v_{t-2}],S_{t-2}[v_{t-2},v_t], P^-_\ell[v_t,v_{t-1}],S_{t-1}[v_{t-1},v_j]$$
is a cycle with chord $v_{t-2}v_{t-1}$, a contradiction. 
If $j=t$, then 
$$P_\ell[v_0,v_{t-1}],S_{t-1}[v_{t-1},v_t],C_1[v_t,v_0]$$ 
is a cycle with chord $v_{t-1}v_t$, a contradiction.  
Thus, for each $1\leq i\leq k$, there exists some vertex $v\in P_i(a_i, b_i)$ 
such that $\deg_{F_k}(v)=2$. 

Next consider $F_1=P_1\cup C$. 
We assume that an orientation of $C$ is given from $a_1$ to $b_1$. 
Then $C[a_1,b_1],P^-_1[b_1,a_1]$ is a cycle in $F_k$. 
By the above result, there exists some vertex $x\in C(b_1,a_1)$ with $\deg _{F_k}(x)=2$.  
Similarly, since $P_1[a_1,b_1],C[b_1,a_1]$ is a cycle in $F_k$, 
there exists some vertex $x'\in C(a_1,b_1)$ with $\deg _{F_k}(x')=2$. 
This completes the proof of Proposition \ref{stem1}. 
\end{proof}

\begin{prop}
\label{n/3stems}
Every non-chorded biconnected graph $H$ of order $n$ has at least $(n-2)/3 +2$ stem vertices. 
\end{prop}

\begin{proof}
Let $C$ and $P_1, \ldots, P_k$ be a cycle and paths satisfying the conclusions of Proposition \ref{decomp}.  
Then by Proposition \ref{stem1},  there exist at least $k+2$ stem vertices in $H$. 
Also, since $H$ is biconnected, every vertex in $H$ is either a stem vertex or a branch vertex. 
Now consider the endpoints of $P_i$ for each $1\leq i\leq k$. 
By Proposition \ref{decomp}, there exist at most $2k$ branch vertices in $H$. 
Thus there exist at least $n-2k$ stem vertices in $H$.
Consequently, the number of stem vertices in $H$ is at least 
$\max\{k+2, n-2k\}$, which is always at least $(n-2)/3+2$.
\end{proof}

\noindent
{\bf Definition.}
A \emph{biconnected component} in a graph is a maximal biconnected subgraph.  
In this paper, we do not consider a single edge to be a biconnected component, 
and we handle these edges separately.
Every cycle in a graph is contained in exactly one biconnected component.  
The following intuitive proposition is shown in \cite{HarPri}.

\begin{prop}
[\rm Harary, Prins \cite{HarPri}] 
\label{equivclass}
If $B_1,B_2$ are distinct biconnected components in a graph, 
then $E(B_1) \cap E(B_2) = \emptyset$.
\end{prop}

\begin{prop}
\label{graphdecomp}
Let $k\geq 1$ be an integer, and 
let $H$ be a non-chorded connected graph containing $k$ biconnected components. 
Then $E(H)$ can be decomposed into
\begin{itemize}
\item a sequence of non-chorded biconnected components $B_1, \ldots, B_k$, and
\item a sequence of edge-disjoint paths $P_2, \ldots, P_\ell$ 
$($some of which might be just a single vertex$)$
with $\ell \geq k$, where the endpoints of $P_i$ are $a_i, b_i$ for each $2\leq i\leq \ell$,
\end{itemize}
so that there exists a sequence of induced subgraphs $F_1, F_2, \ldots, F_\ell$ of $H$ with
the following properties: 

\vspace{0.2cm}

\noindent
$($i$)$ $F_1 = B_1$, 

\vspace{0.2cm}

\noindent
$($ii$)$ for each $2\leq i\leq k$, $F_i = F_{i-1} \cup P_{i} \cup B_{i}$, 
$V(P_i) \cap V(F_{i-1}) = \{a_i\}$, $V(P_i) \cap V(B_i) = \{b_i\}$, 
and $V(F_{i-1}) \cap V(B_i) = \emptyset$ unless $a_i=b_i$, in which case
$V(F_{i-1}) \cap V(B_i) = \{a_i\}$, 

\vspace{0.2cm}

\noindent
$($iii$)$ for each $k+1\leq i\leq \ell$, $F_i = F_{i-1} \cup P_{i}$, 
$V(P_i) \cap V(F_{i-1}) = \{a_i\}$, $\deg_H(b_i)=1$, $|P_i| \geq 2$, and  

\vspace{0.2cm}

\noindent
$($iv$)$ $F_\ell = H$.
\end{prop}


\begin{proof}
Since $H$ is non-chorded, every biconnected component in $H$ must be non-chorded. 
Choose any biconnected component in $H$ to be $F_1=B_1$ (satisfying (\emph{i})).
We claim that $|V(B) \cap V(F_{i-1})| \leq 1$ for any biconnected component $B$ 
in $H \setminus E(F_{i-1})$ and for each $2\leq i\leq k$. 
For some $2\leq i\leq k$, suppose that there exists a biconnected component $B$ 
in $H \setminus E(F_{i-1})$ with $|V(B) \cap V(F_{i-1})| \geq 2$. 
Then for some $u,v \in V(B) \cap V(F_{i-1})$, there exists a path $Q_1$ 
from $u$ to $v$ in $F_{i-1}$ and a path $Q_2$ from $u$ to $v$ in $B$ such that 
$Q_1 \cup Q_2$ forms a cycle $Q$. 
This cycle $Q$ is in $H$. 
Thus $Q$ is contained in some biconnected component $B'$. 
Since $Q_1$ is in $F_{i-1}$, it is edge-disjoint from $B$, 
$Q$ is not in $B$ and $B' \neq B$. 
But $B$ and $B'$ share some edge of $Q_2$, contradicting Proposition \ref{equivclass}. 
Thus the claim holds. 

First suppose that there exists a biconnected component $B$ in $H \setminus E(F_{i-1})$ with 
$V(B) \cap V(F_{i-1}) = \{v\}$ for some vertex $v$. 
In this case, let $B_i=B$, $P_i=v$, 
and $F_{i} = F_{i-1} \cup P_i \cup B_i$, with $a_i=b_i=v$.
Next suppose that all biconnected components in $H \setminus E(F_{i-1})$ are 
vertex-disjoint from $F_{i-1}$. 
Let $B_i$ be a biconnected component in $H \setminus E(F_{i-1})$ such that a path 
from $B_i$ to $F_{i-1}$ in $H$ is edge-disjoint from every other biconnected component in 
$H \setminus E(F_{i-1})$, and let this path be $P_i$. 
Since $H$ is connected, such a $B_i, P_i$ exist. 
Let $F_{i} = F_{i-1} \cup P_i \cup B_i$, $V(P_i) \cap V(F_{i-1})=\{a_i\}$, and 
$V(P_i) \cap V(B_i)=\{b_i\}$. 
Thus (\emph{ii}) is satisfied. 

Clearly $F_k$ is a connected graph containing all the cycles in $H$, 
and $H \setminus E(F_k)$ is a forest. 
Then there exists no path $P$ in $H \setminus E(F_k)$ with both endpoints in $V(F_k)$, 
otherwise $F_k \cup P$ would contain a cycle not in $F_k$. 
If $E(H) \setminus E(F_{i-1})\ne \emptyset$, then do the following:
Select some edge $e \in E(H) \setminus E(F_{i-1})$ that is incident to a leaf vertex $v$ in $H$. 
Let $P_i$ be a path from $v=b_i$ to $V(F_{i-1})$ with $V(P_i) \cap V(F_{i-1})=\{a_i\}$. 
Let $F_i = F_{i-1} \cup P_i$. 
Since $P_i$ contains edge $e$, $|P_i| \geq 2$, 
and since $v=b_i$ is a leaf in $H$, $\deg_H(b_i)=1$, satisfying (\emph{iii}). 

Since $H$ is finite, there exists some $\ell \geq k$ for which 
$E(H) \setminus E(F_{\ell}) = \emptyset$, satisfying (\emph{iv}). 
\end{proof}

Now we finally prove Lemma \ref{indep}. 

\vspace{0.25cm}

\noindent 
{\bf Proof of Lemma \ref{indep}.} 
If $H$ is acyclic, then applying Lemma \ref{tree} $(ii)$ to each connected component of $H$ gives the result. 
Thus we may assume that $H$ has at least one cycle. 
Hence $H$ contains a biconnected component. 
Let $B_1, \ldots, B_k$ and $P_2, \ldots, P_\ell$ be a decomposition of 
$E(H)$ into biconnected components and paths as described by the conclusion of 
Proposition \ref{graphdecomp} with the corresponding subgraphs $F_1, \ldots, F_\ell$ in $H$.
For each $B_i$, $1\leq i\leq k$, 
let $L_i = \lbrace v\in V(B_i) : \deg_{B_i}(v) \leq 2\rbrace$.  
By Proposition \ref{n/3stems}, each $L_i$ has order at least $(|B_i|-2)/3+2$.
Let $S = \lbrace v\in V(H) : \deg_{H}(v) \leq 2\rbrace$.
We will show that $|S| \geq |H|/6$. 
First, let $S_i = \lbrace v\in V(F_i) : \deg_{F_i}(v) \leq 2\rbrace$ for each $1\leq i\leq k$, and 
we claim the following. 
\setcounter{claim}{0}
\begin{claim}
For each $1\leq i\leq k$, $|S_i|\geq |F_i|/5+2$. 
\label{L_i}
\end{claim}

\begin{proof}
First suppose $i=1$. 
Then recall $F_1 = B_1$.  
If $|B_1|\geq 5$, then $|S_1| = |L_1| \geq (|B_1|-2)/3+2 \geq |F_1|/5+2$.
If $|B_1| \leq 4$, then $B_1$ is a 3-cycle or a 4-cycle, 
since these are the only biconnected components on at most 4 vertices. 
Then clearly $|S_1| \geq |F_1|/5+2$. 

Next suppose $2\leq i\leq k$. 
Then recall $F_i = F_{i-1} \cup P_i \cup B_i$,
and assume by inductive assumption that $F_{i-1}$ contains a set $S_{i-1}$ of vertices
of degree at most 2, where $|S_{i-1} | \geq |F_{i-1}|/5+2$.
We have the following two cases. 

\vspace{0.3cm} 

\noindent
\textbf{Case 1.} For some $2\leq i\leq k$, $|P_i| =1$.  

Then $a_i=b_i$. 
By Proposition \ref{graphdecomp}\,$(ii)$
$V(F_{i-1}) \cap V(B_i) = \{a_i\}$. 
Thus $|F_i| = |F_{i-1}| + |B_i| -1$. 
While $a_i$ may have degree 2 in each of $F_{i-1}$, $B_i$ separately, 
it has degree greater than 2 in $F_i$. 
Thus 
\begin{align}
|S_i| &\geq (|S_{i-1}|-|\{a_i\}|)+(|L_i|-|\{a_i\}|) \nonumber \\
&\geq \left(\frac{|F_{i-1}|}{5}+2 \right)+\left(\frac{|B_i|-2}{3}+2\right)-2 \nonumber \\
&= \frac{|F_{i-1}| + |B_i| -1}{5}+\frac{2|B_i|+23}{15} \nonumber \\
&= \frac{|F_i|}{5}+\frac{2|B_i|+23}{15}.  
\label{case1}
\end{align}
If $|B_i|\geq 4$, then, by (\ref{case1}), we have $|S_i| \geq |F_{i}|/5+2$. 
Thus we may assume that $|B_i|\leq 3$. 
Then $B_i$ is a 3-cycle and $|L_i|=3$, in which case the inequality is easily shown. 

\vspace{0.3cm} 

\noindent
\textbf{Case 2.} For some $2\leq i\leq k$, $|P_i| \geq 2$.  

Then $a_i \neq b_i$. 
By Proposition \ref{graphdecomp}\,$(ii)$

$V(F_{i-1}) \cap V(B_i) = \emptyset$. 
Thus 
\begin{align*}
|F_i| &= |F_{i-1}| + |B_i| + |P_i|-|\{a_i, b_i\}| \\
&= |F_{i-1}| + |B_i| + |P_i|-2.
\end{align*}
Note that $\deg_{P_i}(v)\leq 2$ for each $1\leq i\leq \ell$ and every vertex $v\in V(P_i)$.  
While $a_i, b_i$ may have degree 2 in each of $F_{i-1}$, $B_i$  
or $P_i$ separately, they have degree greater than 2 in $F_i$. 
Thus 
\begin{align}
|S_i| &\geq (|S_{i-1}|-|\{a_i\}|) + (|L_i|-|\{b_i\}|) + (|P_i| - |\{a_i, b_i\}|) \nonumber \\ 
&\geq \left(\frac{|F_{i-1}|}{5}+2 \right) + \left(\frac{|B_i|-2}{3}+2\right) + |P_i| - 4 \nonumber \\
&= \frac{|F_{i-1}|+|B_i|+|P_i|-2}{5}+\frac{2|B_i|+12|P_i|-4}{15} \nonumber \\
&= \frac{|F_{i}|}{5}+\frac{2|B_i|+12|P_i|-4}{15}. 
\label{case2}
\end{align}
Note that $|P_i| \geq 2$. 
If $|B_i|\geq 5$, then, by (\ref{case2}), we have $|S_i| \geq |F_{i}|/5+2$. 
Thus we may assume that $|B_i|\leq 4$. 
Then $B_i$ is a 3-cycle or a 4-cycle, and $|L_i|=3$ or 4. 
In either case, the inequality is again easily shown.
\end{proof}

In particular, Claim~\ref{L_i} shows 
\begin{align}
|S_k|\geq|F_k|/5+2.
\label{S_k F_k/5+2}
\end{align}

Let $t = |S_k \cap \bigcup_{i=k+1}^\ell a_i|$. 
Enumerate the components $T_1, T_2, \ldots, T_w$ of $\bigcup_{i=k+1}^\ell \langle V(P_i)\rangle$, 
and note that $w\geq t$.
Clearly 
\begin{align}
t \leq |S_k|. 
\label{T S_i}
\end{align}

\begin{claim}
We have $|S| \geq |H|/6$. 
\label{L}
\end{claim}

\begin{proof}
Each component $T_i$, $1\leq i\leq w$, is a tree, so by Lemma~\ref{tree} $(i)$, it has at least $|T_i|/2+1$ 
vertices of degree at most 2.  
Each component contains exactly one vertex $v\in V(F_k)$, while the rest are in $H - F_k$, 
and this one vertex $v$ may have degree at least 2 in $F_k$, 
so the number of vertices of degree at most 2 in $H - F_k$ is
\begin{align*}
|S \cap (H - F_k)| \geq \sum_{i=1}^w \frac{|T_i|}{2} = 
\sum_{i=1}^w \left(\frac{|T_i|-1}{2} +\frac{1}{2}\right) = \frac{|H| - |F_k|}{2} +\frac{w}{2} \geq \frac{|H| - |F_k|+ t}{2}.
\end{align*}
Also $|S \cap F_k| = |S_k| - t$. 
Then
\begin{align}
|S| = |S \cap F_k| + |S \cap (H - F_k)| \geq |S_k| - t+\frac{|H| - |F_k|+ t}{2}= \frac{|H| - |F_k| - t}{2} + |S_k|. 
\label{|S|}
\end{align}
Combining (\ref{S_k F_k/5+2}), (\ref{T S_i}) and (\ref{|S|}) gives
\begin{align*}
|S| \geq \frac{|H| - |F_k| + |S_k|}{2}\geq \frac{|H|}{2} - \frac{2|F_k|}{5}+1. 
\end{align*}
Since $|S| \geq |S_k|$, by (\ref{S_k F_k/5+2}),  
\begin{align*}
|S| \geq \frac{|F_k|}{5}+2.
\end{align*}
Thus $|S| \geq \max\big\lbrace|H|/2 - 2|F_k|/5+1, |F_k|/5+2\big\rbrace$, which is at least $|H|/6$ for all values of $|F_k|$. 
\end{proof}

We claim that $\langle S\rangle$ is a forest or $H$ is a cycle.  
Suppose $\langle S\rangle$ is not a forest.  
Then $\langle S\rangle$ contains a cycle $C$. 
If $H=C$, then the claim holds. 
Thus $H\ne C$, that is, $V(H)\setminus V(C)\ne \emptyset$. 
Note that $\deg_H(v) \leq 2$ for each $v \in S$. 
Since $H$ is connected by the assumption, we get a contradiction. 
Thus the claim holds. 
If $\langle S\rangle$ is a forest, then it is bipartite.  
Since $|S| \geq |H|/6$ by Claim \ref{L}, 
there exists an independent subset $I \subseteq S$ of order at least $(|H|/6)/2 = n/12$. 
If $H$ is a cycle, then clearly Lemma \ref{indep} also is true.  
This completes the proof of Lemma \ref{indep}. 
\qed

\section{Other Lemmas}

\hspace*{0.5cm} 
In this section, we state several known lemmas that will be used in the proof of our main result.  
Note that a {\it minimal} set of $r$ vertex-disjoint cycles $C_1,\ldots, C_r$ is a set with 
$|\bigcup_{i=1}^r C_i|$ as small as possible.

\begin{lem}[\cite{GHK2}] 
Let $r\geq 1$ be an integer, and let  
$\mathscr{C} = \{C_1, \ldots, C_r\}$ be a minimal set of $r$ 
vertex-disjoint chorded cycles in a graph $G$. 
If $|C_i|\geq 7$ for some $1\leq i\leq r$, then $C_i$ has at most two chords. 
Furthermore, if the $C_i$ has two chords, then these chords must be crossing.
\label{max two chords} 
\end{lem}

\begin{lem}[\cite{GHK2}] 
Let $r\geq 1$ be an integer, and let $\mathscr{C} = \{C_1,\ldots, C_r\}$ be 
a minimal set of $r$  vertex-disjoint chorded cycles in a graph $G$. 
Then $\deg_{C_i}(x)\leq 4$ for any $1\leq i\leq r$ and any 
$x\in V(G)-\bigcup_{i=1}^r V(C_i)$.  Furthermore, for some $C\in \mathscr{C}$ 
and some $x\in V(G)-\bigcup_{i=1}^r V(C_i)$,  
if $\deg_C(x) = 4$, then $|C|=4$, and if $\deg_C(x) = 3$, then $|C|\leq 6$. 
\label{lemma:C is K4 or 6-cycle}
\end{lem}

\begin{lem}[\cite{GHK2}] 
Suppose there exist at least five edges connecting two vertex-disjoint 
paths $P_1$ and $P_2$ with $|P_1\cup P_2|\geq 7$. 
Then there exists a chorded cycle in $\left<P_1 \cup P_2\right>$ 
not containing at least one vertex of $\left<P_1\cup P_2\right>$.  
\label{lem: paths 5 edges}
\end{lem}

\section{Proof of Theorem \ref{thm:main}}

\hspace*{0.5cm} 
Suppose Theorem \ref{thm:main} does not hold. 
We first consider the case where $k=1$. 
Then $n \geq 12t+13$ and $\sigma_t(G) \geq 2t+1$. 
Noting $\lceil n/12\rceil \geq t+2$, by Lemma \ref{indep}, 
$G$ contains an independent set $I$ of order $t$ with each vertex of $I$ having degree at most $2$ in $G$. 
Then $\deg_G(I)\leq 2t$, a contradiction. 
Thus we assume $k\geq 2$.  
Let $G$ be an edge-maximal counter-example.  
If $G$ is complete, then $G$ contains $k$ vertex-disjoint chorded cycles. 
Thus we may assume $G$ is not complete.  
Let $xy\not\in E(G)$ for some $x, y\in V(G)$, and define $G'=G+xy$, the graph 
obtained from $G$ by adding the edge $xy$. 
By the edge-maximality of $G$, $G'$ is not a counter-example.  
Thus $G'$ contains $k$ vertex-disjoint chorded cycles $C_1, \ldots, C_k$.  
Without loss of generality, we may assume $xy\not\in\bigcup_{i=1}^{k-1}E(C_i)$, 
that is, $G$ contains $k-1$ vertex-disjoint chorded cycles.  
Over all sets of $k-1$ vertex-disjoint chorded cycles, 
choose $C_1, \ldots , C_{k-1}$, where
$\mathscr{C}=\bigcup_{i=1}^{k-1}C_i$ and $H=G-\mathscr{C}$, such that:   

\vspace{0.2cm}

(A1) $|\mathscr{C}|$ is as small as possible,  

\vspace{0.2cm}

(A2) subject to (A1), $comp(H)$ is as small as possible, and

\vspace{.2cm}

(A3) subject to (A1) and (A2), the number of $K_4$'s in $\mathscr{C}$ is as large as possible. 

\medskip

We may also assume $H$ does not contain a chorded cycle, otherwise, 
$G$ contains $k$ vertex-disjoint chorded cycles, a contradiction. 
Theorem \ref{thm:main} holds by Theorems 1-4 for all $t \leq 4$. 
Thus we also assume $t \geq 5$. 

\setcounter{claim}{0}
\setcounter{equation}{0}

\noindent
\begin{claim}
$H$ has an order at least $12t+13$.
\label{claim:12t+13}
\end{claim}

\begin{proof}
Suppose this claim fails to hold, that is, suppose $|H| \leq 12t+12$. 
First we prove the following subclaim. 
\begin{subclaim}
For each $1\leq i\leq k-1$, $|C_i|\leq 10t-1$. 
\label{10t-1}
\end{subclaim} 

\begin{proof}
Suppose Subclaim \ref{10t-1} fails to hold, that is, $|C_i|\geq 10t$ for some $1\leq i\leq k-1$. 
Without loss of generality, let $|C_1|\geq |C_2|\geq \cdots \geq |C_{k-1}|$. 
In fact, let $|C_1|=st+r\geq 10t\geq 50$, with $s\geq 10$ and $0\leq r \leq t-1$. 

\begin{sub-subclaim}
For $s \geq 10$, the cycle $C_1$ contains $s$ vertex-disjoint sets 
$X_1, X_2, \ldots, X_s$ each with $t$ independent vertices such that 
$\deg_{C_1}(\bigcup_{i=1}^s X_i) \leq 2st + 4$. 
\label{2st+4}  
\end{sub-subclaim}

\begin{proof} 
For any $st$ vertices of $C_1$, their degree sum in $C_1$ is at most $2st + 4$, 
since by Lemma \ref{max two chords}, $C_1$ has at most two chords. 
Thus, it only remains to show that $C_1$ contains $s$ vertex-disjoint sets of $t$ 
independent vertices each. 
Recall $|C_1| = st + r\geq 10t$. 
Start anywhere on $C_1$ and label the first $st$ vertices of $C_1$ with labels 1 through $s$ 
in order, starting over again with 1 after using label $s$. 
If $r\geq 1$, then label the remaining $r$ vertices of $C_1$ with the labels $s+1,\ldots, s+r$. 
The labeling above yields $s$ vertex-disjoint sets of $t$ vertices each, 
where all the vertices labeled with 1 are one set, all the vertices labeled with 2 are another set, and so on. 
Given this labeling, any vertex in $C_1$ has a different label than 
the vertex that precedes it on $C_1$ and the vertex that succeeds it on $C_1$. 
Let $C_0$ be the cycle obtained from $C_1$ by removing all chords. 
Then the vertices in each of the sets are independent in $C_0$. 
Thus, the only way vertices in the same set are not independent in $C_1$ is if the endpoints 
of a chord of $C_1$ were given the same label. 
Note any vertex labeled $i$ is distance at least $s\geq 10$ in $C_0$ from any other vertex labeled $i$. 
Thus, if a vertex and the neighbor preceding it on $C_1$ or the neighbor succeeding it on $C_1$ 
have their labels exchanged, then the vertices in each of the classes are independent in $C_1$. 

\vspace{0.25cm}

\noindent
\textbf{Case 1.} No chord of $C_1$ has endpoints with the same label. 

Then there exist $s$ vertex-disjoint sets of $t$ independent vertices each in $C_1$. 

\vspace{0.25cm}

\noindent
\textbf{Case 2.} Exactly one chord of $C_1$ has endpoints with the same label. 

Recall $C_1$ contains at most two chords, and 
if $C_1$ contains two chords, then these chords must be crossing. 
Since $|C_1|\geq 50$, even if $C_1$ contains two chords, each chord has an endpoint such that 
one of the endpoint's neighbors in $C_1$ is not an endpoint of the other chord. 
Choose such an endpoint of the chord whose endpoints were assigned the same label, 
and exchange the label of this vertex for its non-endpoint neighbor. 
The vertices in each of the resulting classes are still independent in $C_1$, 
and now no chord of $C_1$ has endpoints with the same label. 
Thus there exist $s$ vertex-disjoint sets of $t$ independent vertices each in $C_1$. 

\vspace{0.25cm}

\noindent
\textbf{Case 3.} Two chords of $C_1$ each have endpoints with the same label. 

In this case, note two chords are crossing. 
Suppose an endpoint of one chord of $C_1$ is adjacent to an endpoint of the other chord on $C_1$. 
Now exchange the labels of these adjacent endpoints. 
Then the vertices in each of the resulting classes are still independent in $C_1$, 
and now no chord of $C_1$ has endpoints with the same label. 
Thus there exist $s$ vertex-disjoint sets of $t$ independent vertices each in $C_1$. 

Next suppose no endpoint of one chord of $C_1$ is adjacent to an endpoint of the other chord on $C_1$. 
Let $x_1x_2, y_1y_2$ be the two distinct chords of $C_1$. 
Since the two chords are crossing, without loss of generality, we may assume 
$x_1, y_1, x_2, y_2$ are in that order on $C_1$, and the label of $x_1$ is 1. 
Then the label of $x_1^+$ is 2. 
Now we exchange the labels of $x_1$ for $x_1^+$, that is, the label of $x_1$ is 2 and  
the label of $x_1^+$ is 1. 
Next we exchange the labels of $y_2$ for $y_2^-$. 
Note $y_2\ne x_1^-$ by our assumption that no endpoint of one chord of $C_1$ is adjacent to an 
endpoint of the other chord on $C_1$. 
Thus, the vertices in each of the resulting classes are independent in $C_1$,  
and no chord of $C_1$ has endpoints with the same label. 
Hence there exist $s$ vertex-disjoint sets of $t$ independent vertices each in $C_1$, 
completing the proof of Subclaim \ref{2st+4}.  
\end{proof}

Recall that, by assumption, $|H| \leq 12t+12$ and $|C_1|\geq 50$. 
Let $X_1,X_2,\ldots, X_s$ be as in Subclaim \ref{2st+4}, 
and let $\mathcal{X}=\bigcup_{i=1}^s X_i$. 
Further, note that $\deg_{C_1}(v)\leq 2$ for every $v\in V(H)$ or a shorter chorded cycle would
exist by Lemma \ref{lemma:C is K4 or 6-cycle}, contradicting (A1). 
Thus 
\begin{align}
|E(H,C_1)| \leq 2(12t+12)
\label{2(12t-1)}. 
\end{align}

First suppose that $k=2$. 
Then $C_1$ is the only cycle in $\mathscr{C}$. 
By Subclaim \ref{2st+4},   
\begin{align*}
|E(C_1,H)| &\geq \deg_G(\mathcal{X})-\deg_{C_1}(\mathcal{X}) \\
&\geq s(3kt-t+1)-(2st+4)\\
&= s(6t-t+1)-(2st+4)\\
&= 3st+s-4, 
\end{align*}
but since $s\geq 10$ and $t\geq 5$, 
we see that $3st+s-4\geq 30t+6>2(12t+12)$, contradicting (\ref{2(12t-1)}). 
Thus we may assume that $k\geq 3$. 
Then, by Subclaim \ref{2st+4} and (\ref{2(12t-1)}), 
\begin{align}
|E(\mathcal{X}, \mathscr{C}-C_1)| &=
\deg_{G}(\mathcal{X})-\deg_{C_1}(\mathcal{X})-\deg_H(\mathcal{X})\nonumber \\
&\geq s(3kt-t+1)-(2st+4)-2(12t+12)\nonumber \\ 
&=3kst-3st+s-24t-28. 
\label{3stk-}
\end{align}
Since $s\geq 10$, we have $3st \geq 30t =24t+6t$. 
Thus 
\begin{align}
24t &\leq 3st-6t.
\label{-3st}
\end{align}
By (\ref{3stk-}) and (\ref{-3st}), we have 
\begin{align*}
3kst-3st+s-24t-28 &\geq 3kst-3st+s-(3st-6t)-28  \\
&= 3st(k-2)+s+6t-28 \\
&\geq 3st(k-2)+12.
\end{align*}
Thus $|E(\mathcal{X}, C')|>3st$ for some $C'$ in $\mathscr{C}-C_1$. 
Let $h=\max\{\deg_{C'}(v) : v \in \mathcal{X}\}$. 
Let $v^{*}\in \mathcal{X}$ with $\deg_{C'}(v^{*})=h$. 
Since $|\mathcal{X}|=st$, if $h\leq 3$, then $|E(\mathcal{X}, C')|\leq 3st$, a contradiction. 
Thus we may assume that $h\geq 4$. 
By the maximality of $C_1$, $|C'| \leq |C_1|=st+r$. 
It follows that $h=\deg_{C'}(v^{*}) \leq |C'| \leq st+r$. 
Recall $s\geq 10$, $t\geq 5$ and $0\leq r \leq t-1$. 
Then 
\begin{align}
|E(\mathcal{X}-\{v^{*}\}, C')| &\geq (3st+1)-\deg_{C'}(v^{*}) \geq (3st+1)-(st+r) \nonumber \\
&=2st+1-r \geq 2st+1-(t-1) \nonumber \\
&=2st-t+2 \nonumber \\
&\geq 97. 
\label{97 edges out}
\end{align}
Since $h=\deg_{C'}(v^{*})\geq 4$, let $v_1,v_2,v_3,v_4$ be neighbors of $v^{*}$ in that order on $C'$. 
These vertices partition $C'$ into four intervals $C'[v_i,v_{i+1})$ for each $1\leq i \leq 4$, where $v_5=v_1$. 
By (\ref{97 edges out}), there exist at least 97 edges from $C_1-v^{*}$ to $C'$. 
Thus some interval clearly receives at least 25 of these edges. 
Without loss of generality, say $C'[v_4,v_1)$ is such an interval. 
Then, by Lemma \ref{lem: paths 5 edges}, $\left<(C_1-v^{*}) \cup C'[v_4,v_1)\right>$  
contains a chorded cycle not containing at least one vertex of  
$\left<(C_1-v^{*}) \cup C'[v_4,v_1)\right>$.  
Also, $v^{*},C'[v_1,v_3],v^{*}$ is a cycle with chord $v^{*}v_2$,  
and it uses no vertices from $ C'[v_4,v_1)$. 
Thus we have two shorter vertex-disjoint chorded cycles in $\left<C_1 \cup C' \right>$, 
contradicting (A1).  
Hence Subclaim \ref{10t-1} holds. 
\end{proof}

Now as $n\geq (10t-1)(k-1)+12t+13$ and $|\mathscr{C}|\leq (10t-1)(k-1)$ by Subclaim \ref{10t-1}, 
we have $|H|\geq 12t+13$, a contradiction.  
This completes the proof of Claim \ref{claim:12t+13}.  
\end{proof}

\smallskip

By Claim \ref{claim:12t+13}, $|H| \geq 12t+13$. 
Noting $\lceil |H|/12\rceil \geq t+2$, by Lemma \ref{indep}, 
there exists an independent set $I^*$ of order $t + 2$ in $H$ such that the degree in $H$ 
of each vertex of $I^*$ is at most 2.   
We now select an independent set $I$ of order $t$ from $I^*$ as follows.  
If $H$ is connected, we select any subset $I$ of order $t$.  
If $H$ is not connected, then each component has a longest path with
endpoints of degree at most 2 in $H$ (or else the component contains a chorded cycle).  
If two of these endpoints are in $I^*$, we select at least two of them, say $s_1$ and $s_2$,
from different components of $H$ to be in $I$.   
Note that each of $s_1$ and $s_2$ is not a cut-vertex for its component.
If $s_1$ and $s_2$ (one or both) are not in $I^*$, then they might have adjacencies in $I^*$.  
We can remove the at most two adjacencies of say
$s_1$ from $I^*$, and place $s_1$ in $I^*$.   
We can do the same for $s_2$ if necessary.   
Then $I^*$ still contains at least $t$ independent vertices with degree at most 2 in $H$.
We select a subset $I$ of order $t$ in $I^*$ that contains both $s_1$ and $s_2$.  
Note that
\begin{align*}
\deg_{\mathscr{C}}(I) & = \deg_G(I) - \deg_H(I)\\
  & \geq (3kt - t + 1) - 2t\\
  & = 3kt -3t + 1\\
  & = 3t(k-1) + 1.
\end{align*}

Therefore, there exists a cycle $C$ in $\mathscr{C}$ such that
$I$ sends at least $3t+1$ edges to $C$.  Thus, by Lemma \ref{lemma:C is K4 or 6-cycle},
since no vertex of $H$ sends more than four edges to a cycle of $\mathscr{C}$,
we see that the degree sequence $D$ of edges from $I$ to $C$ is of the form 
$(4, 4, 4, 4, \ldots)$, $(4, 4, 4, \ldots)$, $(4, 4, 3, \ldots, 3, 2)$ or $(4, 3, \ldots, 3)$. 
Note that if $D=(4, 4, 4, \ldots)$, then $D=(4, 4, 4, 3, \ldots)$, that is, 
$D$ contains at least one 3, or $D=(4, 4, 4, 2, 2)$ for $t=5$. 
Further, since any of these degree sequences contains at least one 4, 
by Lemma \ref{lemma:C is K4 or 6-cycle} we see that $|C|=4$.  
In fact, $C$ induces a $K_4$, otherwise, the vertex of degree 4 along with a triangle in $C$ 
would produce a $K_4$, contradicting (A3). 
Let $C=w_1, w_2, w_3, w_4, w_1$.  

If $D$ has at least two 4's and at least two 3's, 
then it is simple to construct two vertex-disjoint chorded 4-cycles 
from $C$ and these vertices of $I$, as the two vertices of degree 3 are
adjacent to the ends of an edge of $C$ and the two vertices of degree 4
are adjacent to the ends of a different independent edge of $C$.
This produces two vertex-disjoint chorded cycles, 
implying $G$ contains $k$ vertex-disjoint chorded cycles, a contradiction. 
Thus we have only to consider the two cases where $D=(4, 4, 4, 2, 2)$ and $D=(4, 3, \ldots, 3)$. 

First consider $D=(4, 4, 4, 2, 2)$.   
Let $z_1$ be a vertex of $I$ with degree 2 to $C$ and $z_2, z_3, z_4$ be the vertices of $I$ with degree 4.
Without loss of generality, we may assume that $w_1,w_2\in N_C(z_1)$. 
Then $z_1,w_2,z_2,w_1,z_1$ is a cycle with chord $w_1w_2$. 
Also, $z_3,w_3,z_4,w_4,z_3$ is a second cycle with chord $w_3w_4$, 
implying $G$ contains $k$ vertex-disjoint chorded cycles, a contradiction. 

Next consider $D=(4, 3, \ldots , 3)$. 
Let $\deg_C(z_0)=4$ and $\deg_C(z_i)=3$ for each $1\leq i\leq 4$. 
First we prove that 
\begin{align}
H \text{ has no component with one vertex of degree 4 and at least three vertices of degree } 3.
\label{deg4333}
\end{align}
Suppose not, that is, $H$ has a component $H_0$ containing $z_i$ for each $0\leq i\leq 3$. 
Since $H_0$ is connected, there exists a path $P$ from $z_0$ to $z_i$ for some $1\leq i\leq 3$. 
Without loss of generality, we may assume that $i=1$ and $P$ contains neither $z_2$ nor $z_3$. 
Since $\deg_C(z_i)=3$ for each $i\in \{2,3\}$, we may assume that $w_1,w_2\in N_C(z_i)$. 
Then $z_2,w_2,z_3,w_1,z_2$ is a cycle with chord $w_1w_2$. 
Since $\deg_C(z_1)=3$, without loss of generality, we may assume that $w_3\in N_C(z_1)$. 
Then $P[z_0,z_1],w_3,w_4,z_0$ is a second cycle with chord $z_0w_3$, a contradiction. 
Thus (\ref{deg4333}) holds. 

Therefore, we assume that $H$ is not connected, that is, $comp(H)\geq 2$. 
Let $H_1,H_2,\ldots ,H_{comp(H)}$ be the components of $H$. 
Note that it is sufficient to consider the case where each component of $H$ has at least one 
vertex contained in the degree sequence $D=(4, 3, \ldots , 3)$. 
Without loss of generality, for each $i\in \{1,2\}$, we may assume that $s_i\in V(H_i)$ and $\deg_C(s_1)\geq \deg_C(s_2)$. 
Recall, for each $i\in \{1,2\}$, $s_i$ is not a cut-vertex for $H_i$.

\vspace{0.25cm}

\noindent
\textbf{Case 1.} For each $i\in \{1,2\}$, $\deg_C(s_i)=3$. 

\vspace{0.25cm}

In this case, without loss of generality, we may assume that $s_i=z_i$ for each $i\in \{1,2\}$. 

\vspace{0.25cm}

\noindent
{Subcase 1.} Suppose $comp(H)=2$. 

\vspace{0.25cm}

Without loss of generality, we may assume that $z_0\in V(H_1)$. 
By (\ref{deg4333}), we may assume that $z_4\in V(H_2)$. 
For each $i\in \{1,2\}$, since $\deg_C(s_i)=3$, we may assume that $w_1,w_2\in N_C(s_i)$. 
Then $C'=s_1,w_2,s_2,w_1,s_1$ is a 4-cycle with chord $w_1w_2$. 
Since $\deg_C(z_4)=3$, without loss of generality, we may assume that $w_3\in N_C(z_4)$. 
Since $\deg_C(z_0)=4$, $w_4\in N_C(z_0)$.  
Then there exists a path $z_0,w_4,w_3,z_4$ connecting $H_1$ and $H_2$. 
Replacing $C$ in $\mathscr{C}$ by $C'$, we consider the new $H'$. 
Note that $H_i-s_i$ is connected for each $i\in \{1,2\}$. 
Then $comp(H')\le comp(H)-1$. 
This contradicts (A2). 

\vspace{0.25cm}

\noindent
{Subcase 2.} Suppose $comp(H)\geq 3$. 

\vspace{0.25cm}

\noindent
{Subcase 2.1.}  For some $i\in \{1,2\}$, $z_0\in V(H_i)$. 

Without loss of generality, we may assume that $z_0\in V(H_1)$, and $z_4\in V(H_3)$ 
by our assumption that each component of $H$ has at least one vertex contained in the degree 
sequence $D=(4, 3, \ldots , 3)$. 
By the same arguments as Subcase 1, we can reduce the number of components of $H$, 
a contradiction. 

\vspace{0.25cm}

\noindent
{Subcase 2.2.} For some $i\in \{1,2, \ldots, comp(H)\}-\{1,2\}$, $z_0\in V(H_i)$. 

Without loss of generality, we may assume that $z_0\in V(H_3)$. 
Now consider the cycle $C'$ as in Subcase 1. 
If $z_3\in V(H_i)$ for some $i\in \{1,2, \ldots, comp(H)\}-\{3\}$, 
then we apply the same arguments as Subcase 1. 
Thus we may assume that $z_3\in V(H_3)$. 
Since $\deg_C(z_3)=3$, without loss of generality, we may assume that $w_3\in N_C(z_3)$. 
Since $H_3$ is connected, there exists a path $P$ from $z_0$ to $z_3$. 
Then $P[z_0,z_3],w_3,w_4,z_0$ is a second cycle with chord $z_0w_3$, a contradiction. 

\vspace{0.25cm}

\noindent
\textbf{Case 2.} Suppose $\deg_C(s_1)=4$ and $\deg_C(s_2)=3$. 

\vspace{0.25cm}

In this case, note that $s_1=z_0$. 
Without loss of generality, we may assume that $s_2=z_1$. 

\vspace{0.25cm}

\noindent
{Subcase 1.} Suppose $comp(H)=2$. 

\vspace{0.25cm}

\noindent
{Subcase 1.1.} For some $2\le i\le 4$, $z_i\in V(H_1)$.  

Without loss of generality, we may assume that $z_2\in V(H_1)$.  
Since $\deg_C(z_2)=3$ and $\deg_C(s_2)=3$, $N_C(z_2)\cap N_C(s_2)\ne \emptyset$. 
Without loss of generality, we may assume that $w_1\in N_C(z_2)\cap N_C(s_2)$. 
Since $\deg_C(s_1)=4$, $C'=s_1,w_2,w_3,w_4,s_1$ is a 4-cycle with chord $s_1w_3$. 
Replacing $C$ in $\mathscr{C}$ by $C'$, we consider the new $H'$. 
Note that $H_1-s_1$ is connected.  
Then $comp(H')\le comp(H)-1$. 
This contradicts (A2). 

\vspace{0.25cm}

\noindent
{Subcase 1.2.} For each $2\le i\le 4$, $z_i\in V(H_2)$.  

Since $\deg_C(s_2)=3$, without loss of generality, we may assume that $w_i\in N_C(s_2)$ 
for each $1\le i\le 3$. 
If $w_4\in N_C(z_i)$ for some $2\le i\le 4$, then we apply the same arguments as Subcase 1.1. 
Thus we may assume that $N_C(s_2)=N_C(z_i)$ for each $2\le i\le 4$.  
Then $C'=s_2,w_1,w_4,w_2,s_2$ is a 4-cycle with chord $w_1w_2$. 
Replacing $C$ in $\mathscr{C}$ by $C'$, we consider the new $H'$. 
Note that $H_2-s_2$ is connected. 
Since $w_3\in N_C(s_1)\cap N_C(z_2)$, $comp(H')\le comp(H)-1$. 
This contradicts (A2). 


\vspace{0.25cm}

\noindent
{Subcase 2.} Suppose $comp(H)\geq 3$. 

\vspace{0.25cm}

Without loss of generality, we may assume that $z_2\in V(H_3)$ 
by our assumption that each component of $H$ has at least one vertex contained in the degree sequence $D$.  
By the same arguments as Subcase 1.1, we can reduce the number of components of $H$, 
a contradiction. 

\vspace{0.25cm}

This completes the proof of Theorem \ref{thm:main}. 
\qed

\section{Conclusion}

\hspace*{0.5cm} 
We believe that Lemma \ref{indep} may be improved to guarantee a larger independent set of low-degree 
vertices in every non-chorded connected graph. In particular, we conjecture the following.

\begin{conj}
If $H$ is a non-chorded connected graph of order $n$, then $H$ contains an independent set 
$I$ of order at least $n/6$ with each vertex of $I$ having degree at most $2$ in $H$. 
\end{conj}
This $1/6$ proportion of vertices would be best possible, 
as we demonstrate with two examples $G_1$ and $G_2$.

First, define the graph $H$ with 6 vertices to be the graph containing a 5-cycle $x_1, x_2, x_3, x_4, x_5, x_1$ 
and where the sixth vertex $x_6$ is adjacent to $x_2$ and $x_5$.  
To form $G_1$, take $k$ copies of $H$ called $H^1, H^2, \ldots, H^k$. 
Let $x_i^j\in V(H^j)$ with $1\leq i\leq 6$ and $1\leq j\leq k$, 
and let $x_6^jx_1^{j+1}\in E(G_1)$ for each $1\leq j\leq k-1$. 
Aside from $H^1$ and $H^k$, each copy of $H$ has exactly two vertices of degree 2, 
and only one of these can be included in the independent set $I$.  
Each of $H^1$ and $H^k$ have two independent vertices of degree 2, so $|I|=n/6+2$.

Second, construct $G_2$ by starting with a triangle, and for each of its vertices, 
connect it by an edge to a new triangle. 
Then for each vertex of degree 2 in this graph, connect it by an edge to a new triangle.  
Repeat this process $k$ times.  
In $G_2$, every vertex of degree 2 is adjacent to another vertex of degree 2, 
so only one of each pair can be in $I$.  
By adding a triangle adjacent to each vertex of degree 2 in the pair, 
we can increase the size of $I$ by 1, 
and we have added 6 vertices.  
That means the limit $$\lim_{k \to \infty} \frac{\vert I \vert}{n} = \frac{1}{6},$$
so no larger proportion than $1/6$ of the vertices of $G_2$ can be in $I$.
\medskip

We also note the following easy-to-prove facts about graphs with no chorded cycles.  We did not use these facts in our proof of Theorem~\ref{thm:main} but they may be of interest to the reader.

\begin{fact} 
If $G$ is a graph of order n with no chorded cycles, then there exists an ordering of the
vertices of $G$ such that each vertex has at most two neighbors preceding it in this ordering. 
Further $G$ is a tripartite graph. 
\label{order}
\end{fact}

\begin{fact} 
If $G$ is a graph of order $n$ containing no chorded cycles, then $|E(G)|\le 2n-4$. 
\label{2n-4}
\end{fact}

\vspace{0.35cm}

\noindent
{\bf Acknowledgments.} 
The authors would like to thank Ariel Keller Rorabaugh for the helpful insights she provided.
The second author is supported by the Heilbrun Distinguished Emeritus 
Fellowship from Emory University Emeritus College. 
The third author is supported by JSPS KAKENHI Grant Number JP19K03610.

\end{document}